\documentclass[12pt, reqno, twoside, letterpaper]{amsart}

\usepackage{paperstyle}
\usepackage{booktabs}

\makeatletter
\def\section{\@startsection{section}{1}%
  \z@{.7\linespacing\@plus\linespacing}{.5\linespacing}%
  {\normalfont\large\scshape\centering}}
\makeatother

\newif\ifdraft
\draftfalse
\ifdraft
  \long\def\td#1{\ccm{#1}}
  \long\def\tdpar#1{\par\noindent\ccm{#1}\par}
\else
  \long\def\td#1{}
  \long\def\tdpar#1{}
\fi

\newcommand{\G}{\mathcal{G}}
\newcommand{\Orb}{\mathcal{O}}
\newcommand{\eq}[1]{\e_q\!\left(#1\right)}
\newcommand{\bl}{\bar}

\title[Products of two elements of Granville's model]{Products of two
elements of\\ Granville's random model in residue classes}

\author{William D. Banks}
\address{Department of Mathematics\\
University of Missouri\\
Columbia, MO 65211\\
USA}
\email{bankswd@missouri.edu}


\date{September 29, 2026}

\begin{document}

\begin{abstract}
Erd\H{o}s, Odlyzko and S\'ark\"ozy conjectured that every
reduced residue class modulo $q$ contains a product of two
primes not exceeding $q$; their conjecture remains open even under
the GRH. We prove that its analogue
holds, unconditionally and in the wider range $q^{\Theta}$ for
any fixed $\Theta>\tfrac34$, in Granville's random model of the
primes. Almost surely, every reduced class modulo every
sufficiently large modulus is representable, and the number of
representations is asymptotic to its expected order. The choice of the
model is decisive; the corresponding statement in
the random prime model of Banks, Ford and Tao fails with positive
probability, an elementary local defect that Granville's
discard step removes.
\end{abstract}
\maketitle

\section{Introduction and statement of results}\label{sec:intro}

\subsection{Background}

For which $N$ can every reduced residue class modulo $q$ be
represented by a product of $k$ primes not exceeding $N$? Given
$q,k,N\in\N$, let
\[
\cS_{q,k}(N)\defeq\big\{a\in(\Z/q\Z)^\times:p_1\cdots
p_k\equiv a\bmod q
\text{~for some primes~}p_1,\ldots,p_k\le N\big\}.
\]
Linnik~\cite{Lin1,Lin2} showed that every reduced class modulo
$q$ contains a prime $p\ll q^{L}$ for some absolute
constant $L$.\footnote{The current record $L=5$ is due to Xylouris~\cite{Xy}.}
In other words, $\cS_{q,1}(N)=(\Z/q\Z)^\times$ once $N/q^{L}$ is large
enough. 

Erd\H{o}s, Odlyzko and S\'ark\"ozy (EOS)~\cite{EOS} posed the
product analogue of Linnik's problem and conjectured that
$\cS_{q,2}(q)=(\Z/q\Z)^\times$ for all large $q$, that is, 
every class $a\in(\Z/q\Z)^\times$ contains a product $p_1p_2$ of two
primes $p_1,p_2\le q$.\footnote{The conjecture is often credited
to Erd\H{o}s alone, as in~\cite{MT}.} Their paper gives
partial results assuming strong zero-free regions for Dirichlet
$L$-functions.

The EOS conjecture remains open even under the Generalized Riemann Hypothesis (GRH).
To see why, consider the natural approach of
using Dirichlet characters modulo $q$ to detect the congruence
$p_1p_2\equiv a\bmod q$. The principal character $\chi_0$
produces the main term $\pi(q)^2/\varphi(q)$, whereas by Parseval the mean
square of the sums $\sum_{p\le q}\chi(p)$ over $\chi\ne\chi_0$
equals $(1+o(1))\,\pi(q)$, so bounding the error terms by absolute values
yields a total that exceeds the main term by the factor
$\varphi(q)/\pi(q)\gg\log q/\log\log q$. The loss comes from the sum over the
characters, not from any single character. The square-root saving in
each individual sum, which is all that GRH provides, does not recover this loss.

The EOS conjecture is also resistant to sieve methods,
which can produce products of two almost primes in a given class
but not products of two primes, since integers with exactly two
prime factors are the standard example of the parity problem in
sieve theory.

Unconditional results are known for several weaker versions
of the conjecture. Friedlander, Kurlberg and Shparlinski~\cite{FKS} obtained results
on average over $q$ and $a$, and with the primes replaced by more
tractable deterministic sequences. Garaev~\cite{Ga} sharpened several
averaged variants via bounds for Kloosterman sums with prime arguments, and
Shparlinski~\cite{Sh13,Sh19} represented every class by products
of primes and almost primes in various ranges. In a wider setting,
the congruence $n_1n_2\equiv a\bmod q$ with the variables confined
to prescribed sequences, is the theory of modular hyperbolas
surveyed in~\cite{Sh}, to which our Lemma~\ref{lem:hyperbola} below
properly belongs.

Another line of inquiry retains the primes but weakens the
conclusion, covering a positive proportion of the classes by
products of two primes or all of $(\Z/q\Z)^\times$ by products of
more than two. Table~\ref{tab:results} lists the known results.
Walker's lemma~\cite{Wa} that $\cS_{q,1}(q)$ lies in no proper
coset of $(\Z/q\Z)^\times$ underlies all later work, and
in~\cite{Sz,MT,KMT} the obstacle is a quadratic character
$\psi\bmod q$ with $\psi(p)=-1$ for almost all primes $p\le q$,
as a Siegel zero would produce.

\begin{table}[ht]
\centering
{\renewcommand{\arraystretch}{1.3}%
\begin{tabular}{@{}lll@{}}
\toprule
Result & Range of $q$ & Reference \\
\midrule
$|\cS_{q,2}(q)|\ge(1/64+o(1))\,\varphi(q)$ & $q$ prime & \cite{Wa} \\
$\cS_{q,48}(q)=(\Z/q\Z)^\times$ & $q$ prime & \cite{Wa} \\
$\cS_{q,3}(q^{3/2+\eps})=(\Z/q\Z)^\times$ & $q$ large & \cite{RW,RSS,BRS} \\
$|\cS_{q,2}(q)|\ge(3/8+o(1))\,\varphi(q)$ & $q$ cube-free & \cite{Sz} \\
$\cS_{q,6}(q)=(\Z/q\Z)^\times$ & $q$ large prime & \cite{Sz} \\
$\cS_{q,3}(q^{6/5+\eps})=(\Z/q\Z)^\times$ & all $q$ & \cite{Sz} \\
$|\cS_{q,2}(q)|\ge(2/3-\eps)\,\varphi(q)$ & $q$ large & \cite{MT} \\
$\cS_{q,3}(q)=(\Z/q\Z)^\times$ & $q$ large cube-free & \cite{MT} \\
$\cS_{q,3}(q^{2+\eps})=(\Z/q\Z)^\times$ & $q$ smooth or non-exceptional prime & \cite{KMT} \\
\bottomrule\smallskip
\end{tabular}}
\caption{Products of $k$ primes in reduced classes modulo $q$;
$\eps>0$ is fixed and ``large'' means $q\ge q_0(\eps)$ where
$\eps$ appears.}
\label{tab:results}
\end{table}

Over function fields the analogous problem has been settled for large fields.
Sawin~\cite{Sa} proved square-root cancellation for
factorization functions in squarefree progressions in
$\F_{\ell}[t]$, and as Xie observes in~\cite{Xi}, the result
\cite[Lemma~9.14]{Sa} implies that every reduced class modulo a
squarefree modulus of large degree $n$ contains a product of two
monic irreducibles of degree $n$ once $\ell$ exceeds an absolute
constant; Xie's own proof, via Katz's equidistribution, needs
$\ell$ large in terms of $n$. Over number fields, Deshouillers,
Gun, Ramar\'e and Sivaraman~\cite{DGRS} represented every narrow
ray class by a product of three prime ideals of small norm, and
Xie~\cite{Xi2} extended the transference method of~\cite{MT} to
ray class groups.

\subsection{Results}
The present note replaces the primes by a random set. The
analysis of conjectures about the primes inside a probabilistic
model, as a rigorous statement about the model, goes back to
Cram\'er and has recently been carried out for prime
gaps~\cite{BFT} and for Gilbreath's conjecture~\cite{CHT}.
In this spirit, we prove the analogue of the EOS conjecture,
unconditionally and in a wider range, for Granville's random
model $\G$ of the primes~\cite{Gr}. To our knowledge, the
conjecture has not previously been studied in random sets.

The model is recalled
precisely in \S\ref{sec:model}; for the statement below it suffices to
say that $\G$ discards every integer with a prime factor
$\le T$ and includes any surviving~$n$ independently with
probability $\prod_{p\le T}(1-p^{-1})^{-1}/\log n$. The
threshold $T$ is left free in~\cite{Gr}, subject to growth at
least a fixed power of $\log u$ at scale $u$. In this paper,
we take $T=A(\euX)$ on each dyadic block $(\euX,2\euX]$, where
$\euX$ is a power of two, and $A(\euX)=(\log \euX)^{1-o(1)}$.
Our main result is the following.

\begin{theorem}\label{thm:main}
Fix $\Theta\in(\tfrac34,1]$ and a choice of the parameter $A$ in the
definition of $\G$. For each $q\ge 16$, let $\euX_q$ be the largest
power of two not exceeding $\tfrac12 q^{\Theta}$, and put
\[
c_q\defeq\frac{\varphi(q)}{q}
\sprod{p\le A(\euX_q)\\ p\,\mid\, q}(1-p^{-1})^{-2}.
\]
Then, almost surely, there is a number $q_0$ such that for every
modulus
$q\ge q_0$ and all $a\in(\Z/q\Z)^\times$, there exist distinct
$g_1,g_2\in\G\cap(\euX_q,2\euX_q]$ with
$g_1g_2\equiv a\bmod q$. Moreover,
if $N_q(a)$ denotes the number of ordered pairs of distinct
$g_1,g_2\in\G\cap(\euX_q,2\euX_q]$ satisfying that congruence,
then almost
surely we have
\[
\max_{a\in(\Z/q\Z)^\times}
\Big|\,N_q(a)\cdot\frac{q\log^2\euX_q}{c_q\,\euX_q^2}-1\Big|
\longrightarrow 0
\qquad (q\to\infty).
\]
\end{theorem}

Four comments provide context for the theorem.

First, since $(\euX_q,2\euX_q]\subseteq[1,q^{\Theta}]$, the
first assertion gives solutions to $g_1g_2\equiv a\bmod q$ with
$g_1,g_2\le q^{\Theta}$. We emphasize that $q$ ranges over
\emph{all} moduli, and that the constant~$c_q$ is the
natural local correction. 

Second, the exponent $\tfrac34$ is sharp for the method, and the
reason is a deterministic one. The proof reduces to counting lattice
points on the hyperbola $n_1n_2\equiv a\bmod q$ in boxes of side
length $N\asymp q^{\Theta}$, where completion and the Weil--Estermann
bound give an error $O(q^{1/2+o(1)})$ against a main term of size
roughly $N^2/q$; these balance at $N=q^{3/4}$. Any saving on
incomplete Kloosterman sums would lower the exponent, but below
$\Theta=\tfrac12$ the statement certainly fails.

Third, the probabilistic content of our argument is minimal.
The count is a sum of independent Bernoulli variables over orbits of
$n\mapsto a\bar n$, so the probability that there are no
representations can be precisely estimately, and Borel--Cantelli finishes
the proof. Bennett's inequality~\cite{Be} enters only for the
asymptotic count.

Finally, the choice of the model matters. The proof applies
verbatim with $T=1$, so the conclusion also holds for Cram\'er's
prime model, the constant $c_q$ degenerating to $\varphi(q)/q$;
Although the model succeeds, it does not see the bias of the
primes to small moduli. For the sieve model $\cR$ of~\cite{BFT},
in which a single uniformly random residue class $a_p \bmod p$ is deleted
for every small prime $p$, the analogue of the EOS conjecture \emph{fails}
with positive probability when $\gcd(q,6)\ne 1$.
Indeed, if $a_2\ne 0$, an event of probability
$\tfrac12$, then every element of $\cR$ is even and no reduced
class modulo an even $q$ is represented at all; a similar situation
occurs when $a_3\ne 0$. These comparisons illuminate the choice of the model:
Granville's deterministic discard step gives $\G$ the small-modulus bias of
the primes while leaving the indicators $Z_n$ independent, and
our proof uses both features.

\subsection{Conventions}

Unless indicated otherwise, summation
variables always run over the positive integers. For $y\in\R$ we write
$\eq{y}\defeq\er^{2\pi i y/q}$ and $\|y\|$ for the distance from $y$
to the nearest integer; $\tau(q)$ and $\omega(q)$ denote the number of
divisors and the number of distinct prime divisors of $q$, respectively,
and $\varphi$ is the Euler totient function. 
For a finite set $S$ we write $|S|$ for its cardinality, and for
an interval $J$ we write $|J|$ for its length. For integers $m$
coprime to $q$, $\bl m$ denotes the inverse of $m$ modulo $q$.
The indicator of an event $E$ is denoted $\ind{E}$. The notations
$F\ll G$ and $F=O(G)$ both mean that $|F|\le c\,G$ holds for some
absolute constant $c>0$; any dependence of the implied constant
is indicated explicitly.

\section{The Granville model}\label{sec:model}

Granville~\cite[pp.~23--24]{Gr} attaches to any threshold
parameter $T$ a sequence $(Z_n)_{n\ge 3}$ of independent random
variables, where $Z_n=0$ whenever $n$ has a prime factor $\le T$,
and otherwise
\[
\PP(Z_n=1)=\prod_{p\le T}\big(1-p^{-1}\big)^{-1}\cdot\frac{1}{\log n}.
\]
The random set is defined by
\[
\G\defeq\{n\ge 3:Z_n=1\}.
\]
The case $T=1$ recovers Cram\'er's model; Granville instead takes
$T$ to be at least a fixed power of $\log u$ at scale $u$.
We work with the following specification. Fix once and for all
\[
A(u)\defeq \log u\cdot\er^{-\sqrt{\log\log u}}\qquad(u\ge\er).
\]
For any power of two $\euX\ge 4$, let $T$ have the value
$A(\euX)$ over the dyadic block $(\euX,2\euX]$. Write
\[
Q_\euX\defeq\prod_{p\le A(\euX)}p,\qquad\text{so that}\quad
\frac{Q_\euX}{\varphi(Q_\euX)}
=\prod_{p\le A(\euX)}\big(1-p^{-1}\big)^{-1}.
\]
The prime number theorem and Mertens' theorem give
\be\label{eq:Qsize}
Q_\euX=\er^{(1+o(1))A(\euX)}=\euX^{o(1)},
\qquad
\frac{Q_\euX}{\varphi(Q_\euX)}\asymp\log A(\euX)\sim\log\log \euX.
\ee
On each block $(\euX,2\euX]$, those integers $n$ for which
$\gcd(n,Q_\euX)>1$
are discarded, and each surviving $n$ is included in $\G$ with
probability
\[
\euP_n\defeq\PP(Z_n=1)
=\frac{Q_\euX}{\varphi(Q_\euX)\log n},
\]
all these inclusion events being jointly independent,
within each block as well as across distinct blocks.

Fix $\Theta\in(\tfrac34,1]$. For a given modulus $q\ge 16$,
let $\euX=\euX_q$ be the largest power of two such
that $2\euX\le q^{\Theta}$; then,
\be\label{eq:xsize}
\tfrac14 q^{\Theta}<\euX\le\tfrac12 q^{\Theta}
\mand
\cI\defeq(\euX,2\euX]\subseteq[1,q^{\Theta}]\subseteq[1,q].
\ee
We abbreviate $A\defeq A(\euX_q)$, $Q\defeq Q_{\euX_q}$. Since
$|\cI|=\euX<q$,
distinct elements of $\cI$ are also distinct modulo $q$.

Now let $a\in(\Z/q\Z)^\times$. For each $n\in\cI$,
there is at most one $m\in \cI$ such that $nm\equiv a\bmod q$;
when it exists, we write $m=\iota_a(n)$, and in this case,
both $n$ and $\iota_a(n)$ are coprime to $q$. Define
\[
\cI_a\defeq\big\{n\in \cI:\iota_a(n)\text{~exists},\ \iota_a(n)\ne n\big\}.
\]
Observe that $\iota_a$ is a fixed-point-free involution of $\cI_a$.
Let $\Orb_a$ be the set of its orbits, each an
unordered pair $\{n,\iota_a(n)\}$ of distinct integers, so
$|\Orb_a|=\frac12|\cI_a|$. For $\omega=\{n,m\}\in\Orb_a$, we define
\[
Z_{\omega}\defeq Z_nZ_m=\ind{n\in\G}\,\ind{m\in\G},
\qquad
\euP_{\omega}\defeq\euP_n\euP_m.
\]
Since $n\ne m$ and neither is discarded, $Z_\omega$ is a Bernoulli
variable of parameter $\euP_{\omega}$. Since distinct orbits are
disjoint sets of integers, the family $(Z_\omega)_{\omega\in\Orb_a}$
is jointly independent. Notice that the count $N_q(a)$ defined in
Theorem~\ref{thm:main} satisfies
\be\label{eq:Nsum}
N_q(a)=\big|\big\{(g_1,g_2)\in(\G\cap \cI)^2:
g_1g_2\equiv a\bmod q,\ g_1\ne g_2\big\}\big|
=2\sum_{\omega\in\Orb_a}Z_{\omega}.
\ee

\section{Lattice points on the modular hyperbola}\label{sec:hyperbola}

Our solitary arithmetic input is the following classical estimate; we
include a proof for completeness. For an extensive account
of such results, see Shparlinski~\cite{Sh}.

\begin{lemma}\label{lem:hyperbola}
Let $q\ge 2$ and $b\in(\Z/q\Z)^\times$,
and let $\cJ_1,\cJ_2$ be real intervals of length at most~$q$. Then
\[
\big|\big\{(m_1,m_2)\in(\cJ_1\times \cJ_2)\cap\Z^2:
m_1m_2\equiv b\bmod q\big\}\big|
=\frac{\varphi(q)}{q^2}\,|\cJ_1||\cJ_2|
+O\big(\tau(q)^2q^{1/2}\log^2q\big),
\]
uniformly for $q$, $b$, and the intervals $\cJ_j$.
\end{lemma}

\begin{proof}
For $j=1$ or $2$, we set
\[
S_j(\ell)\defeq\sum_{m\in \cJ_j\cap\Z}\eq{\ell m}\qquad(\ell\in\Z).
\]
Then $S_j(\ell)$ depends only on the residue class of $\ell$
modulo $q$, and $S_j(0)=|\cJ_j|+O(1)$. For $q\nmid\ell$, summing
the geometric series yields
\be\label{eq:geom}
|S_j(\ell)|\le\tfrac12\|\ell/q\|^{-1},\qquad\text{and so}\quad
\sum_{\ell=1}^{q-1}|S_j(\ell)|\ll q\log q.
\ee

Let $J$ denote the count in the lemma. Since $\gcd(b,q)=1$, every
pair $(m_1,m_2)$ counted by $J$ has $\gcd(m_1,q)=1$,
and $m_2$ is uniquely determined by $m_2\equiv b\,\bl m_1\bmod q$
since $|\cJ_2|\le q$. Using additive characters 
detect the latter congruence, we have
\dalign{
J&=\ssum{m_1\in\cJ_1\cap\Z\\\gcd(m_1,q)=1}\sum_{m_2\in \cJ_2\cap\Z}~\cdot~
\frac{1}{q}\sum_{\ell\in\Z/q\Z}\eq{\ell(m_2-b\,\bl m_1)}\\
&=\frac{1}{q}\sum_{\ell\in\Z/q\Z}S_2(\ell)
\ssum{m_1\in\cJ_1\cap\Z\\\gcd(m_1,q)=1}\eq{-\ell\,b\,\bl m_1}
=\frac{\Phi}{q}\,S_2(0)+\frac1q\sum_{\ell=1}^{q-1}S_2(\ell)\,{\widetilde S}(\ell),
}
where
\[
\Phi\defeq\big|\{m\in \cJ_1\cap\Z:\gcd(m,q)=1\}\big|
\mand
{\widetilde S}(\ell)\defeq\ssum{m\in \cJ_1\cap\Z\\ \gcd(m,q)=1}\eq{-\ell\,b\,\bl m}.
\]
M\"obius inversion leads to the estimate
\[
\Phi=\sum_{d\,\mid\, q}\mu(d)\Big(\frac{|\cJ_1|}{d}+O(1)\Big)
=\frac{\varphi(q)}{q}\,|\cJ_1|+O(\tau(q)),
\]
and since $S_2(0)=|\cJ_2|+O(1)$ and $|\cJ_j|\le q$, the main term satisfies
\[
\frac{\Phi}{q}\,S_2(0)
=\frac{\varphi(q)}{q^2}\,|\cJ_1||\cJ_2|+O(\tau(q)).
\]
To complete the proof, it remains to show that
\be\label{eq:showthis}
\frac1q\sum_{\ell=1}^{q-1}\big|S_2(\ell)\big|\,\big|{\widetilde S}(\ell)\big|
\ll\tau(q)^2\,q^{1/2}\log^2q.
\ee

We complete the sum ${\widetilde S}(\ell)$ in a standard way,
as follows. Fix $\ell$ with $1\le \ell\le q-1$, and put
$u\defeq-\ell\,b$, which satisfies $\gcd(u,q)=\gcd(\ell,q)$. Writing
\[
K(u,v;q)\defeq\ssum{r\in\Z/q\Z\\ \gcd(r,q)=1}\eq{u\bl r+vr}
\qquad(v\in\Z/q\Z)
\]
for the Kloosterman sum, the orthogonality of the additive characters
yields
\[
\frac1q\sum_{v\in\Z/q\Z}\eq{-vm}K(u,v;q)
=\ssum{r\in\Z/q\Z\\ \gcd(r,q)=1}\eq{u\bl r}\ind{r\equiv m\bmod q}
=\begin{cases}
\eq{u\bl m}&\quad\hbox{if $\gcd(m,q)=1$},\\
0&\quad\hbox{otherwise}.
\end{cases}
\]
Summing over all $m\in \cJ_1\cap\Z$, we conclude that
\be\label{eq:cmpl}
{\widetilde S}(\ell)=\frac1q\sum_{v\in\Z/q\Z}S_1(-v)\,K(u,v;q).
\ee
By the Weil--Estermann bound~\cite{Es} (see
also~\cite[Cor.~11.12]{IK}), the Kloosterman sum satisfies
\be\label{eq:kloostbd}
|K(u,v;q)|\le\tau(q)\gcd(u,v,q)^{1/2}q^{1/2}.
\ee
We also record the elementary bound
\be\label{eq:elembd}
\sum_{v=1}^{q-1}\|v/q\|^{-1}\gcd(v,q)^{1/2}
\ll\tau(q)\,q\log q.
\ee
In the expression \eqref{eq:cmpl}, the term $v=0$ contributes
\[
\frac1qS_1(0)\,K(u,0;q)\ll\tau(q)\gcd(\ell,q)^{1/2}q^{1/2}
\]
by \eqref{eq:kloostbd} and the trivial bound $|S_1(0)|\le q+1$,
since $\gcd(u,0,q)=\gcd(\ell,q)$. Using \eqref{eq:geom}
and~\eqref{eq:kloostbd}, the bound $\gcd(u,v,q)\le\gcd(v,q)$,
and then \eqref{eq:elembd}, the terms $v\ne 0$ contribute
\dalign{
\frac1q\sum_{v=1}^{q-1}S_1(-v)\,K(u,v;q)
&\ll\frac1q\sum_{v=1}^{q-1}
\|v/q\|^{-1}\tau(q)\gcd(u,v,q)^{1/2}q^{1/2}\\
&\ll\frac{\tau(q)}{q^{1/2}}\sum_{v=1}^{q-1}
\|v/q\|^{-1}\gcd(v,q)^{1/2}
\ll\tau(q)^2q^{1/2}\log q.
}
Altogether, this shows that
\[
|{\widetilde S}(\ell)|\ll\tau(q)\gcd(\ell,q)^{1/2}q^{1/2}
+\tau(q)^2\,q^{1/2}\log q.
\]
Consequently,
\dalign{
\frac1q\sum_{\ell=1}^{q-1}|S_2(\ell)|\,|{\widetilde S}(\ell)|
&\ll\frac{\tau(q)}{q^{1/2}}\sum_{\ell=1}^{q-1}
|S_2(\ell)|\gcd(\ell,q)^{1/2}
+\frac{\tau(q)^2\log q}{q^{1/2}}\sum_{\ell=1}^{q-1}|S_2(\ell)|,
}
and the required bound \eqref{eq:showthis} follows by applying
\eqref{eq:geom} and \eqref{eq:elembd}.
\end{proof}

\section{The expected number of representations}\label{sec:mean}

Throughout this section, $q$ denotes a given modulus, and we
abbreviate
\[
\euX\defeq\euX_q,\qquad
\cI\defeq(\euX,2\euX],\qquad
A\defeq A(\euX),\qquad
Q\defeq Q_\euX.
\]
\begin{proposition}\label{prop:count}
Uniformly for $q\ge 16$ and $a\in(\Z/q\Z)^\times$,
\[
R(a)\defeq\big|\big\{(n_1,n_2)\in \cI^2:
n_1n_2\equiv a\bmod q,\ \gcd(n_1n_2,Q)=1\big\}\big|
\]
satisfies the estimate
\[
R(a)=\kappa_q\,\frac{\varphi(q)}{q^2}\,\euX^2+O(q^{1/2+o(1)}),
\qquad\text{where}\quad
\kappa_q\defeq\sprod{p\le A\\p\,\nmid\,q}(1-p^{-1})^2.
\]
\end{proposition}

\begin{proof}
Expanding
$\ind{\gcd(n_j,Q)=1}=\sum_{d_j\mid\,\gcd(n_j,Q)}\mu(d_j)$ for
$j=1,2$, we have
\[
R(a)=\sum_{d_1,d_2\,\mid\,Q}\mu(d_1)\mu(d_2)\,T(d_1,d_2),
\]
where
\[
T(d_1,d_2)\defeq\big|\big\{(n_1,n_2)\in \cI^2:
n_1n_2\equiv a\bmod q,\ d_j\mid n_j\big\}\big|.
\]
If a prime $p$ were to divide $\gcd(d_1d_2,q)$, then any pair
counted by
$T(d_1,d_2)$ would have both $p\mid n_1n_2$ and $p\mid n_1n_2-a$,
forcing $p\mid a$, which is not possible since $\gcd(a,q)=1$.
Hence $T(d_1,d_2)=0$ unless $\gcd(d_1d_2,q)=1$. In the latter case,
substituting $n_j=d_jm_j$ transforms the count into that studied in
Lemma~\ref{lem:hyperbola} with intervals
$\cJ_j=(\euX/d_j,2\euX/d_j]$ of
lengths $\euX/d_j\le q$ and residue $b=a\,\overline{d_1d_2}$, which is
coprime to $q$. Applying the lemma, we get  
\[
T(d_1,d_2)=\frac{\varphi(q)}{q^2}\cdot\frac{\euX^2}{d_1d_2}
+O\big(\tau(q)^2q^{1/2}\log^2 q\big).
\]
The divisors $d_j$ of $Q$ are squarefree with all prime factors at
most $A$, so the number of pairs $(d_1,d_2)$ is at most
$4^{\pi(A)}\le\er^{2A}=\euX^{o(1)}\le q^{o(1)}$
by \eqref{eq:Qsize}, and $\tau(q)^2\log^2 q=q^{o(1)}$, so
the total contribution of the error terms is $O(q^{1/2+o(1)})$.
The main terms sum to
\[
\frac{\varphi(q)}{q^2}\,\euX^2
\bigg\{\ssum{d\,\mid\,Q\\ \gcd(d,q)=1}\frac{\mu(d)}{d}\bigg\}^{2}
=\frac{\varphi(q)}{q^2}\,\euX^2
\sprod{p\le A\\ p\,\nmid\,q}(1-p^{-1})^{2}
=\kappa_q\,\frac{\varphi(q)}{q^2}\,\euX^2,
\]
and the proposition is proved.
\end{proof}

The next corollary evaluates the expectation $\EE N_q(a)$. By
\eqref{eq:Nsum}, the count $N_q(a)$ introduced in
Theorem~\ref{thm:main} is twice the sum over $\omega\in\Orb_a$
of the jointly independent Bernoulli variables $Z_\omega$ with
parameters $\euP_\omega$; taking expectations, we have
\be\label{eq:Esum}
\EE N_q(a)=2\sum_{\omega\in\Orb_a}\euP_{\omega}.
\ee

\begin{corollary}\label{cor:mean}
With notation as in Theorem~\ref{thm:main}, the estimate
\[
\EE N_q(a)=\frac{c_q\,\euX_q^2}{q\log^2 \euX_q}
\big(1+O\big((\log q)^{-1}\big)\big)
\]
holds uniformly for $q\ge 16$ and $a\in(\Z/q\Z)^\times$.
In particular, there exists $q_1=q_1(\Theta,A)$ such that
$\EE N_q(a)\ge 2q^{1/2}$ for all $q\ge q_1$ and $a\in(\Z/q\Z)^\times$.
\end{corollary}

\begin{proof}
The ordered pairs counted by $R(a)$ with $n_1=n_2$ satisfy
$n_1^2\equiv a\bmod q$. Since $\euX\le q$, their number is at most the
number of square roots of $a$ modulo $q$, which is at most
$2^{\omega(q)+1}=q^{o(1)}$. Removing them, the remaining pairs are
the ordered versions of the orbits in $\Orb_a$; thus,
$|\Orb_a|=\tfrac12R(a)+O(q^{o(1)})$. For each $n\in \cI$, we have
$\log n=\log\euX+O(1)$; it follows that
\[
\euP_{\omega}=\frac{Q^2}{\varphi(Q)^2\log^2\euX}
\big(1+O\big((\log\euX)^{-1}\big)\big)
\]
for every $\omega\in\Orb_a$. Recalling \eqref{eq:Esum} and
applying Proposition~\ref{prop:count}, we derive the estimate
\[
\EE N_q(a)=\frac{Q^2}{\varphi(Q)^2\log^2\euX}
\bigg\{\kappa_q\,\frac{\varphi(q)}{q^2}\,\euX^2
+O\big(q^{1/2+o(1)}\big)\bigg\}
\big(1+O\big((\log\euX)^{-1}\big)\big).
\]
Now, since
\[
\frac{Q^2}{\varphi(Q)^2}\,\kappa_q
=\sprod{p\le A\\p\,\mid\,q}(1-p^{-1})^{-2},
\]
and $\log\euX\asymp\log q$, the main term of $\EE N_q(a)$ equals
\[
\frac{c_q\,\euX^2}{q\log^2 \euX}
\big(1+O\big((\log q)^{-1}\big)\big),
\]
and the error term is $q^{1/2+o(1)}$ by \eqref{eq:Qsize}. Since
$c_q\ge\varphi(q)/q\gg(\log\log 3q)^{-1}$ and
$\euX>\tfrac14q^{\Theta}$ by
\eqref{eq:xsize}, the main term is at least $q^{2\Theta-1-o(1)}$.
As $2\Theta-1>\tfrac12$, the main term dominates the error term,
and both assertions of the corollary follow.
\end{proof}

Corollary~\ref{cor:mean} is the only place where the hypothesis
$\Theta>3/4$ is used.

\section{Proof of Theorem~\ref{thm:main}}\label{sec:proof}

Both assertions rest on the independence of the orbit variables
$(Z_\omega)_{\omega\in\Orb_a}$. The first follows from the exact
product formula \eqref{eq:exact} for $\PP(N_q(a)=0)$ together with the
Borel--Cantelli lemma; the second combines Bennett's inequality
with another application of Borel--Cantelli.

Let $q\ge q_1$ and $a\in(\Z/q\Z)^\times$. By \eqref{eq:Nsum} and the
independence of the Bernoulli family $(Z_\omega)_{\omega\in\Orb_a}$,
we have
\be\label{eq:exact}
\PP\big(N_q(a)=0\big)=\prod_{\omega\in\Orb_a}\big(1-\euP_{\omega}\big)
\le\exp\Big(-\sum_{\omega\in\Orb_a}\euP_{\omega}\Big)
=\exp\big(-\tfrac12\EE N_q(a)\big)\le\er^{-\sqrt q},
\ee
by Corollary~\ref{cor:mean}. Consequently, for the event
\[
E_q\defeq\big\{N_q(a)=0\text{~for some~}a\in(\Z/q\Z)^\times\big\},
\]
the union bound gives
\[
\sum_{q\ge q_1}\PP(E_q)\le\sum_{q\ge q_1}q\,\er^{-\sqrt q}<\infty,
\]
and the Borel--Cantelli lemma yields the first assertion of the
theorem.

For the second assertion, we begin by proving that
\be\label{eq:outrageous}
\PP\Big(\big|N_q(a)-\EE N_q(a)\big|\ge\delta\,\EE N_q(a)\Big)
\le 2\er^{-\frac13\delta^2\sqrt{q}}
\ee
uniformly for $q\ge q_1$, $a\in(\Z/q\Z)^\times$, and
$0<\delta\le 1$. For this, we use a classical
concentration inequality; see also~\cite[Lemma~3.3]{BFT}.

\begin{lemma}[Bennett's inequality \cite{Be}]\label{lem:Bennett-ineq}
Suppose that $X_1,\ldots,X_n$ are independent random variables
satisfying $\EE X_j=0$ and $|X_j|\le M$ almost surely for each
$j$, where $M>0$. Then
\[
\PP\bigg(\bigg|\sum_{1\le j\le n} X_j\bigg| \ge t\bigg)
\le 2\exp\Big\{-\frac{\sigma^2}{M^2}\,
\sL\Big(\frac{Mt}{\sigma^2}\Big)\Big\}\qquad(t>0),
\]
where $\sigma^2\defeq\sum_j \VV X_j$ and
\[
\sL(u)\defeq\int_1^{1+u}\log t\,dt=(1+u)\log(1+u)-u.
\]
\end{lemma}

We apply the lemma to the random variables
$X_\omega\defeq Z_\omega-\euP_{\omega}$, $\omega\in\Orb_a$, which
are independent, have mean zero, and satisfy $|X_\omega|\le 1$;
thus $M=1$ and, by \eqref{eq:Esum},
\[
\sigma^2=\sum_{\omega\in\Orb_a}\VV X_\omega
=\sum_{\omega\in\Orb_a}\VV Z_\omega
=\sum_{\omega\in\Orb_a}(\euP_{\omega}-\euP_{\omega}^2)
\le\tfrac12\,\EE N_q(a)\eqdef\mu.
\]
If $\sigma^2=0$, the count $N_q(a)$ is deterministic,
so $N_q(a)-\EE N_q(a)=0$, but $\EE N_q(a)>0$
by Corollary~\ref{cor:mean}, so \eqref{eq:outrageous}
holds automatically in this case.

Now suppose that $\sigma^2>0$. Fix $\delta\in(0,1]$ and take 
$t\defeq\delta\mu$. Since
\[
N_q(a)-\EE N_q(a)=2\sum_{\omega\in\Orb_a}X_\omega
\]
by \eqref{eq:Nsum} and \eqref{eq:Esum},
Lemma~\ref{lem:Bennett-ineq} (with $M=1$) gives
\[
\PP\Big(\big|N_q(a)-\EE N_q(a)\big|\ge\delta\,\EE N_q(a)\Big)
\le 2\exp\Big\{-\sigma^2
\sL\Big(\frac{\delta\mu}{\sigma^2}\Big)\Big\}.
\]
The function
$v\mapsto v\,\sL(\delta\mu/v)$ is nonincreasing on $(0,\infty)$,
since its derivative equals
$\log(1+u)-u\le 0$ with $u\defeq\delta\mu/v$. Hence,
as $0<\sigma^2\le\mu$, we have
$\sigma^2\sL(\delta\mu/\sigma^2)\ge\mu\,\sL(\delta)$. Moreover,
$\mu\ge q^{1/2}$ for $q\ge q_1$ by Corollary~\ref{cor:mean},
and 
\[
\sL(\delta)=(1+\delta)\log(1+\delta)-\delta\ge\tfrac13\delta^2
\qquad(0<\delta\le 1),
\]
thus
\[
\exp\Big\{-\sigma^2
\sL\Big(\frac{\delta\mu}{\sigma^2}\Big)\Big\}
\le\exp\{-\mu\sL(\delta)\}\le \er^{-\frac13\delta^2\sqrt{q}},
\]
giving \eqref{eq:outrageous} in this case as well.

Fix $k\in\N$ and take $\delta=1/k$. Summing \eqref{eq:outrageous} over
all $q\ge q_1$ and $a\in(\Z/q\Z)^\times$ gives a convergent
series, so the Borel--Cantelli lemma shows that, almost surely,
\[
\max_{a\in(\Z/q\Z)^\times}
\Big|\frac{N_q(a)}{\EE N_q(a)}-1\Big|\le\frac1k
\qquad\text{for all sufficiently large }q.
\]
Intersecting these almost-sure events over all $k\in\N$,
we conclude that, almost surely,
\[
\max_{a\in(\Z/q\Z)^\times}
\Big|\frac{N_q(a)}{\EE N_q(a)}-1\Big|
\longrightarrow 0
\qquad(q\to\infty).
\]
Since Corollary~\ref{cor:mean} gives
$\EE N_q(a)=(1+O(1/\log q))\,c_q\euX_q^2/(q\log^2\euX_q)$
uniformly in $a$, the second assertion of the theorem follows,
and the proof is complete. \qed

\section*{Acknowledgments}

The author thanks the Max Planck Institute for Mathematics in
Bonn for its support and excellent working conditions during the
preparation of this paper.

The author discloses that the large language model
Claude (Anthropic) was used in the preparation of this paper, in
particular for suggestions on writing and exposition, for literature
searches, and for proof checking. The author takes full responsibility
for the mathematical content.

\end{document}